\documentclass[12pt]{amsart}
\usepackage{bbm,bm}
\usepackage{amssymb}
\usepackage[margin=1.5in, total={6.0in, 8in}]{geometry}
\usepackage{colordvi}
\usepackage{graphicx,xspace}
\usepackage{color}
\usepackage{xspace}

\usepackage[bookmarks=false,colorlinks,linkcolor=blue,citecolor=green]{hyperref}

\newcounter{FNC}[page]
\def\fauxfootnote#1{{\addtocounter{FNC}{2}$^\fnsymbol{FNC}$%
     \let\thefootnote\relax\footnotetext{$^\fnsymbol{FNC}$\Magenta{#1}}}}

\numberwithin{equation}{section}
\newtheorem{theorem}{Theorem}[section]

\newtheorem{lemma}[theorem]{Lemma}

\newtheorem{prop}[theorem]{Proposition}

\newtheorem{defi}[theorem]{Definition}
\newtheorem{exam}[theorem]{Example}
\newtheorem{rema}[theorem]{Remark}

\def\R{\mathbb R}
\def\p{\mathcal P}
\def\KP{{\mathcal K}{\mathcal P}}

\author{Stefan Forcey} \address[S. Forcey]{
    Department of Mathematics\\
    The University of Akron\\
    Akron, OH 44325-4002\\
Financial support for this research was
received from the\\
 Faculty Research Committee of The University of Akron
    }
    \email{sf34@uakron.edu}  \urladdr{http://www.math.uakron.edu/\~{}sf34/}

\author{Logan Keefe} \address[L. Keefe]{
    Department of Mathematics\\
    The University of Akron\\
    Akron, OH 44325-4002
    }

\author{William Sands} \address[W. Sands]{
    Department of Mathematics\\
    The University of Akron\\
    Akron, OH 44325-4002
    }

\title[BME and $\KP$ Facets]{Split-facets for Balanced Minimal Evolution Polytopes and the Permutoassociahedron.}

\begin{document}

\begin{abstract}

Understanding the face structure of the balanced minimal evolution
(BME) polytope, especially its top-dimensional facets, is a
fundamental problem in phylogenetic theory.
 We show that BME polytope has a sub-lattice of its poset of faces which is isomorphic to a quotient of the well-studied
 permutoassociahedron. This sub-lattice corresponds to compatible sets of splits displayed by phylogenetic trees, and
 extends the lattice of faces of the BME polytope found by Hodge, Haws, and Yoshida.
  Each of the maximal elements in our new poset of faces corresponds to a single split of the leaves. Nearly all of these turn out to actually be facets of the BME polytope, a collection of facets which grows exponentially.
\end{abstract}

\maketitle

\section{Introduction}

\emph{Phylogenetics} is the study of the reconstruction of
biological family trees from genetic data. Results from
phylogenetics can inform every facet of modern biology, from natural
history to medicine. A chief goal of biological research is to find
relationships between genes and the functional structures of
organisms.  Knowing degrees of kinship can allow us to decide
whether or not an adaptation in two species is probably an inherited
feature of a common ancestor, and thus help to isolate the roles of
genes common to both.\footnote{Financial support for this research was
received from the Faculty Research Committee of The University of Akron}

Mathematically, a \emph{phylogenetic tree} is a cycle-free graph
with no nodes (vertices) of degree 2, and with a set of distinct
items assigned to the degree one nodes--that is, labeling the
leaves. We study a method called \emph{balanced minimal evolution}.
This method begins with a given set of $n$ items and a symmetric (or
upper triangular) square $n\times n$ \emph{dissimilarity matrix}
whose entries are numerical dissimilarities, or distances, between
pairs of items. From the dissimilarity matrix, often presented as a
vector of distances, the balanced minimal evolution (BME) method
constructs a binary (degree of vertices $\le3$) phylogenetic tree
with the $n$ items labeling the $n$ leaves.

 It is well known that if
a distance vector is created by starting with a given binary tree
$T$ with lengths assigned to each of its edges, and finding the
pairwise distances between leaves just by adding the edge lengths
along the path that connects them, then the tree $T$ is uniquely
recovered from that distance vector. The distance vector (or matrix)
is called \emph{additive} in this case. One recovery process is
called the \emph{sequential algorithm}, described first in
\cite{waterman}. It operates by leaf insertion and is performed in
polynomial time: $O(n^2)$. Another famous algorithm is
\emph{neighbor joining}, which reconstructs the tree in $O(n^3)$
time \cite{Saitou}. It has the advantage of being a greedy algorithm
for the BME problem, when extended to the non-additive case
\cite{Gascuel}.

An alternate method of recovery via minimization was introduced by
Pauplin in \cite{Pauplin} and developed by Desper and Gascuel in
\cite{fastme}. This BME method uses a linear functional on binary
phylogenetic trees $t$ (without edge lengths) defined using the
given distance vector. The output of the function is the length of
the original tree $T$ (assuming that the distance vector was created
from $T.$) The function is minimized when the input tree $t$ is
identical to $T$, as trees without edge lengths. Thus by minimizing
this functional, we recover the original \emph{tree topology}.
 The latter terminology is used to describe two
trees that are identical if we ignore edge lengths. The value of
this approach is that the given distance vector is often corrupted
by missing or incorrect data; but within error bounds we can still
recover the tree topology by the minimization procedure.
Furthermore, the BME method is \emph{statistically consistent} in
that as the distance vector approaches the accuracy of a true tree
$T,$ the BME method's output approaches that tree's topology
\cite{despernew, Atteson1999, Steel2006}.

More precisely: Let the set of $n$ distinct species, or taxa, be
called $S.$ For convenience we will often let $S = [n] =
\{1,2,\dots,n\}.$ Let vector $\mathbf{d}$ be given, having ${n
\choose 2}$ real valued components $d_{ij}$, one for each pair
$\{i,j\}\subset S.$ There is a vector $\mathbf{c}(t)$ for each
binary tree $t$ on leaves $S,$ also having ${n \choose 2}$
components $c_{ij}(t)$, one for each pair $\{i,j\}\subset S.$ These
components are ordered in the same way for both vectors, and we will
use the lexicographic ordering: $\mathbf{d} =
\left<d_{12},d_{13},\dots,d_{1n},d_{23},d_{24},\dots,d_{n-1,n}\right>
$.

We define, following Pauplin \cite{Pauplin}:
 $$c_{ij}(t) = \frac{1}{2^{l_{ij}(t)}}$$
where ${l_{ij}(t)}$ is the number of internal nodes (degree 3
vertices) in the path from leaf $i$ to leaf $j.$

If a phylogenetic tree $T$ with non-negative edge lengths is given,
then we can define the distance vector $\mathbf{d}(T)$ by adding the
edge lengths between each pair of leaves. Then the dot product
$\mathbf{c}(T)\cdot\mathbf{d}(T)$ is equal to the sum of all the
edge lengths of $T,$ a sum which is known as the \emph{tree length}.
$T$ is uniquely determined by $\mathbf{d}(T)$ (unless there are
length zero edges, in which case there is a finite set of trees
determined). Using any other tree $t$ as the input of
$\mathbf{c}(t)$ will give a sub-optimal, larger value for
$\mathbf{c}(t)\cdot\mathbf{d}(T).$

The BME tree for an arbitrary positive vector $\mathbf{d}$ is the
binary tree $t$ that minimizes $\mathbf{d}\cdot\mathbf{c}(t)$ for
all binary trees on leaves $S.$ Now this dot product is the least
variance estimate of treelength, as shown in \cite{despernew}.   The
value of setting up the question in this way is that it becomes a
linear  programming problem. The convex hull of all the vectors
$\mathbf{c}(t)$ for all binary trees $t$ on $S$ is a polytope
BME$(S)$, hereafter also denoted BME($n$) or $\mathcal{P}_n$ as in
\cite{Eickmeyer} and \cite{Rudy}. The vertices of $\mathcal{P}_n$
are precisely the $(2n-5)!!$ vectors $\mathbf{c}(t).$ Minimizing our
dot product over this polytope is equivalent to minimizing over the
vertices, and thus amenable to the simplex method.

In Fig.~\ref{f:2dbmes} we see the 2-dimensional polytope
$\mathcal{P}_4.$ In that figure we illustrate a simplifying choice
that will be used throughout: rather than the original fractional
coordinates $c_{ij}$ we will scale by a factor of $2^{n-2},$ giving
a new vector $\mathbf{x}(t)$ with coordinates:

$$x_{ij}(t)=2^{n-2}c_{ij}(t) = 2^{n-2-{l_{ij}(t)}}.$$

The convex hull of the vectors $\mathbf{x}(t)$ is a combinatorially
equivalent scaled version of the BME polytope, so we refer to it by
the same name. Since the furthest apart any two leaves may be is a
distance of $n-2$ internal nodes, this scaling will result in
integral coordinates for our polytope. The tree $t$ that minimizes
$\mathbf{d}\cdot\mathbf{c}(t)$ will also minimize
$\mathbf{d}\cdot\mathbf{x}(t).$

\begin{figure}[b!]\centering
                  \includegraphics[width=\textwidth]{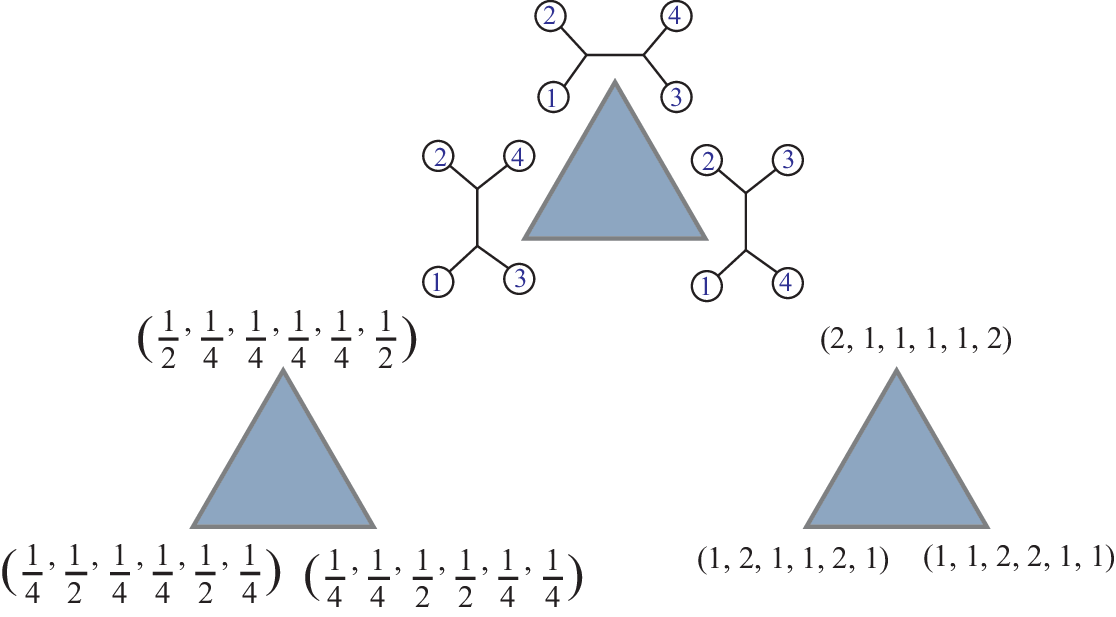}
\caption{The polytope $\mathcal{P}_4$ is a triangle. At the top we
label the vertices with the three binary trees with leaves $1\dots
4$. Each edge shows a nearest-neighbor interchange; for instance the
exchange of leaves 3 and 4 on the bottom edge. At bottom left are
Pauplin's original coordinates and at bottom right are the
coordinates, scaled by $2^{n-2}=4$, which we will
use.}\label{f:2dbmes}
\end{figure}

A \emph{clade} is a subgraph of a binary tree induced by an internal
 (degree three) node and all of the leaves descended from it in a particular
 direction. In other words: given an
 internal node $v$ we choose two of its edges and all of the leaves that
 are connected to $v$ via those two edges. Equivalently, given any
 internal edge, its deletion separates the tree into two clades. Two clades on the same tree
  must be either disjoint or \emph{nested}, one contained in the other. A
 \emph{cherry} is a clade with two leaves. We often refer to a clade
 by its set of (2 or more) leaves.  A pair of \emph{intersecting cherries} $\{a,b\}$
 and $\{b,c\}$ have intersection in one leaf $b$, and thus cannot exist both on the same tree. A \emph{caterpillar} is a tree
 with only two cherries. A
\emph{split} of the set of $n$ leaves for our phylogenetic trees is
a partition of the leaves into two  parts, one part called $S_1$
with $m$ leaves and another $S_2$ with the remaining $n-m$ leaves. A
tree \emph{displays} a split if each part makes up the leaves of a
\emph{clade}. A \emph{facet} of a polytope is a top-dimensional face
of that polytope's boundary, or a co-dimension-1 face. Faces of a
polytope can be of any dimension, from 0 to that of the (improper)
face which is the polytope itself.

\section{New Results}
Our most important new discovery is a large family of facets of the
BME polytope, which we call $split-facets$ in
Theorem~\ref{t:splitfacet}. This collection of facets is shown to
exist for all $n,$ and the number of facets in this family grows
like $2^n.$

In Theorem~\ref{t:faces} we show that any (non-binary) phylogenetic
tree corresponds to a face of $\p_n.$ This allows us to define a map
from the permutoassociahedron to the BME polytope, taking faces to
faces. In Theorem~\ref{t:lattice} we show that this map preserves
the partial order of faces.

In Theorem~\ref{t:cladeface} we show that a special case of these
tree-faces are the clade-faces discovered earlier in \cite{Rudy}. In
Theorem~\ref{t:splitfacet} we show that another special case of
tree-faces is our new class of facets of $\mathcal{P}_n.$

%

\section{Previous results}

Until recently, little was known about the structure of the BME
polytopes, but several newly discovered features were described in
\cite{Huggins} and \cite{Rudy}. The coordinates of the vertices
satisfy a set of $n$ independent equalities, which we will refer to
as the Kraft equalities, after an equivalent description in
\cite{Catan}. For each leaf $i$ we sum the coordinates that involve
it:

$$\sum_{j : j\ne i} x_{ij} = 2^{n-2}.$$
These equalities govern the dimension of the BME polytope,
dim$(\p_n) = {n \choose 2} -n.$

In \cite{Rudy} the authors prove the first description of faces of
the $n^{th}$ balanced minimal evolution polytope $\mathcal{P}_n$.
They find a family of faces that correspond to any set of disjoint
clades. In \cite{forcey2015facets} we show that these clade-faces
are not facets, but instead show several new familys of facets. We
add to that list here with a family of facets that grows
exponentially. (Our results are listed in columns 5--7 of
Table~\ref{facts}.)

 We show in \cite{forcey2015facets} that any pair of
intersecting cherries corresponds to a facet of $\mathcal{P}_n.$ For
each pair of cherries with leaves $\{a,b\}$ and $\{b,c\},$ there is
a facet of $\mathcal{P}_n$ whose vertices correspond to trees that
have either one of those two cherries.

In addition, any caterpillar tree with fixed ends corresponds to a
facet of $\mathcal{P}_n.$ Thus for each pair of species there is a
facet of $\mathcal{P}_n$ whose vertices correspond to trees which
are caterpillars with this pair as far apart as possible. Also shown
in \cite{forcey2015facets}: for $n=5,$ for each necklace of five
leaves there is a corresponding facet which is combinatorially
equivalent to a simplex.\vspace{.1in}

%
%
\begin{table}[hb!]
\begin{tabular}{|c|c|c|c||c|c|c|c|}
\hline
number& dim. & vertices &  facets & facet inequalities & number of & number of \\
of& of $\mathcal{P}_n$&of $\mathcal{P}_n$&of $\mathcal{P}_n$&(classification)& facets &   vertices \\
species&&&&&& in facet\\
\hline \hline
3 & 0 & 1 & 0 &-&-&- \\
\hline 4 & 2 & 3 & 3 & $ x_{ab}\ge 1$ & 3 & 2 \\   \cline{5-7}
&&&& $x_{ab}+x_{bc}-x_{ac} \le 2$ & 3&2\\
\hline
5 & 5 & 15 & 52 &  $x_{ab}\ge 1$ & 10&6 \\
&&&& (caterpillar)&&\\
 \cline{5-7}
&&&&$ x_{ab}+x_{bc}-x_{ac} \le 4$ & 30 & 6\\
&&&& (intersecting-cherry)&&\\
\cline{5-7} &&&&\scriptsize{ $x_{ab}+x_{bc}+x_{cd}+x_{df}+x_{fa}\le
13$} & 12 & 5\\
&&&& (cyclic ordering)&&\\
 \hline
 6 & 9 & 105 & 90262 &  $x_{ab}\ge 1$ & 15 &
24\\
&&&& (caterpillar)&&\\
  \cline{5-7}
 &&&& $~ x_{ab}+x_{bc}-x_{ac} \le 8$ &$60$ & $30$ \\
 &&&& (intersecting-cherry)&&\\
 \cline{5-7}
  &&&& $~ x_{ab}+x_{bc}+x_{ac} \le 16$& 10 & 9\\
&&&& $(3,3)$-split && \\
  \hline \hline
 \rule{0pt}{2.6ex}\rule[-1.2ex]{0pt}{0pt}   $n$ & $\binom{n}{2}-n$  & $(2n-5)!!$ & ? & $x_{ab}\ge 1$ &$\binom{n}{2}$& $(n-2)!$\\
&&&& (caterpillar)&&\\
  \cline{5-7}
 \rule{0pt}{2.6ex}\rule[-1.2ex]{0pt}{0pt}  &&&& $~x_{ab}+x_{bc}-x_{ac} \le 2^{n-3}$ &$\binom{n}{2}(n-2)$ & $2(2n-7)!!$ \\
 &&&& (intersecting-cherry)&&\\
 \cline{5-7} \rule{0pt}{2.6ex}\rule[-1.2ex]{0pt}{0pt} &&&& $x_{ab}+x_{bc}+x_{ac} \le 2^{n-2}$ & $\binom{n}{3}$ & $3(2n-9)!!$\\
  &&&& $(m,3)$-split, $m > 3$ && \\
  \cline{5-7} \rule{0pt}{2.6ex}\rule[-1.2ex]{0pt}{0pt} &&&&
$\displaystyle{\sum_{i,j\in S_1} x_{ij} \le (k-1)2^{n-3}}$ &
$2^{n-1}-\binom{n}{2}$ &
 $(2m-3)!!$\\
 &&&& $(m,k)$-split, & ~~$-n-1$ & $\times(2k-3)!!$ \\
   &&&& $m>2,k>2$ && \\
  \hline
\end{tabular}\caption{Technical statistics for the BME polytopes $\mathcal{P}_n$. The first four columns are found in \cite{Huggins} and
\cite{Rudy}. Our new and recent results are in the last 3 columns.
 The inequalities are given for any $a,b,c,\dots \in [n].$ Note that for
 $n=4$ the three facets are described twice: our inequalities are redundant. \label{facts}}
\end{table}

\section{Connection to the Permutoassociahedron}

The $n^{th}$ permutoassociahedron $\KP_n$, also known as the type-A Coxeter
associahedron, is defined in \cite{kapra}. It is discussed in detail in
\cite{ReinerZiegler1993}, and related to the space of phylogenetic trees in
\cite{bhv}. A face of the permutoassociahedron
 corresponds to an ordered partition of a set $S$ of $n$ elements, whose parts label the
leaves, left to right, of a rooted plane tree. We often use $S =
\{1,\dots,n\}.$ Alternatively we  may use $S = \{1,\dots,n+1\}
-\{r\}$ where $r\in S$ is the label for the root. Bijectively, one
of these labeled plane trees can also be described as a partial
bracketing of an ordered partition, such as
$((\{3\},\{4,5\}),\{2\},\{1,6,7\}).$

The inclusion of faces corresponds to refinement of the
ordered-partition trees: refinement of the tree structure by adding
branches at nodes with degree larger than 3 (so that the collapse of
the added branches returns the original tree) or refinement of the
ordered partition, in which parts of it are further partitioned
(subdivided, with ordering). To display the subdivision, the parts
of the refined partition label the ordered leaves of a new subtree:
a \emph{corolla}, which is a tree with one root, one internal node,
and 2 or more leaves.

 Tree refinement can also be described as
adding parentheses to the bracketing, or subdividing a set in the
bracketing.  A covering relation is either adding a single branch
(pair of parentheses) or subdividing a single part of the partition.
For examples of covering relations,

$((\{3\},\{4,5\}),\{2\},\{1,6,7\}) >
(((\{3\},\{4,5\}),\{2\}),\{1,6,7\})$

and

$((\{3\},\{4,5\}),\{2\},\{1,6,7\}) >
((\{3\},\{4,5\}),\{2\},(\{1\},\{6\},\{7\}))$

or

$((\{3\},\{4,5\}),\{2\},\{1,6,7\}) >
((\{3\},\{4,5\}),\{2\},(\{1,7\},\{6\})).$

The 2-dimensional $\KP_2$ is shown in Fig.~\ref{f:2dkp}. The
3-dimensional $\KP_3$ is shown in Fig.~\ref{f:3dkp}

\begin{figure}[b!]\centering
                  \includegraphics[width=4.25in]{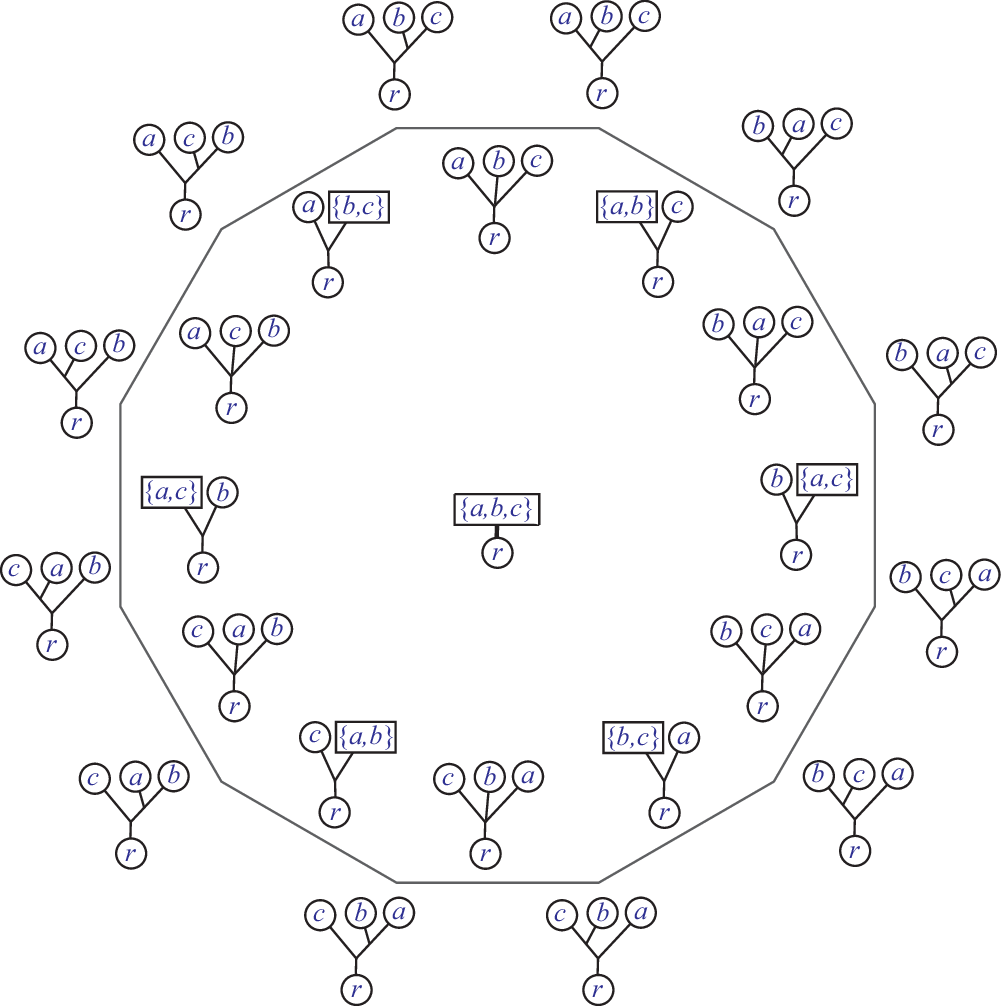}
\caption{The 2-dimensional permutoassociahedron with labeled
faces.}\label{f:2dkp}
\end{figure}

\begin{figure}[b!]\centering
                  \includegraphics[width=3.5in]{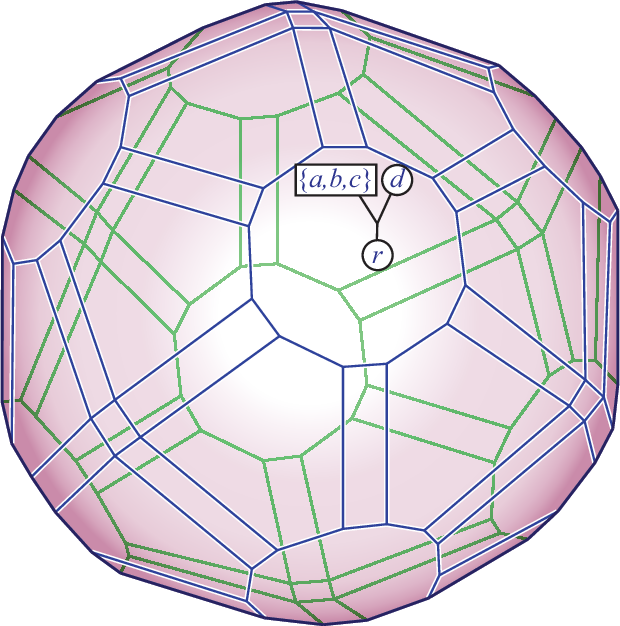}
\caption{The 3-dimensional permutoassociahedron with a labeled
facet. This picture is redrawn from a version in \cite{ReinerZiegler1993}.}\label{f:3dkp}
\end{figure}

There is a straightforward lattice map $\varphi$ from the faces of
$\KP_n$ to a sub-lattice of faces of the BME polytope. Since it
preserves the poset structure, its preimages are a nice set of
equivalence classes.
\begin{defi}
Let $t$ be a plane rooted tree with leaves an ordered partition
$\pi$ of $S.$ First let $t'$ be the tree achieved by replacing
each leaf labeled by part $U\in\pi$ such that $|U|>1$ with a corolla
labeled by the elements of $U.$ This corolla is attached at a new
branch node where the leaf labeled by $U$ was attached. Now let
$\varphi(t)$ be the tree $f(t')$, where $f$ is described as
un-gluing $t'$ from the plane in order to preserve only its
underlying graph.
\end{defi}

An example of the map $\varphi$ is shown in Fig.~\ref{f:phi}. Note
that forgetting the plane structure of $t'$ ensures that the map
$\varphi$ is well-defined. The corolla that replaces each leaf
labeled by $U\in\pi$ is immediately seen as unordered since it is
not fixed in the plane.

\begin{figure}[b!]\centering
                  \includegraphics[width=\textwidth]{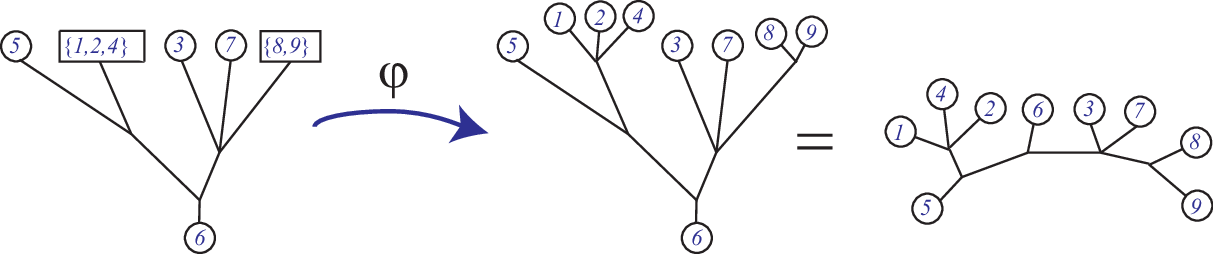}
\caption{Action of the map $\varphi$: the second step shows that we
no longer preserve plane structure or rooted-ness. }\label{f:phi}
\end{figure}

Fig.~\ref{f:2phi} shows the full action of $\varphi$ on the
2-dimensional $\KP_2.$

\begin{figure}[b!]\centering
                  \includegraphics[width=6.5in]{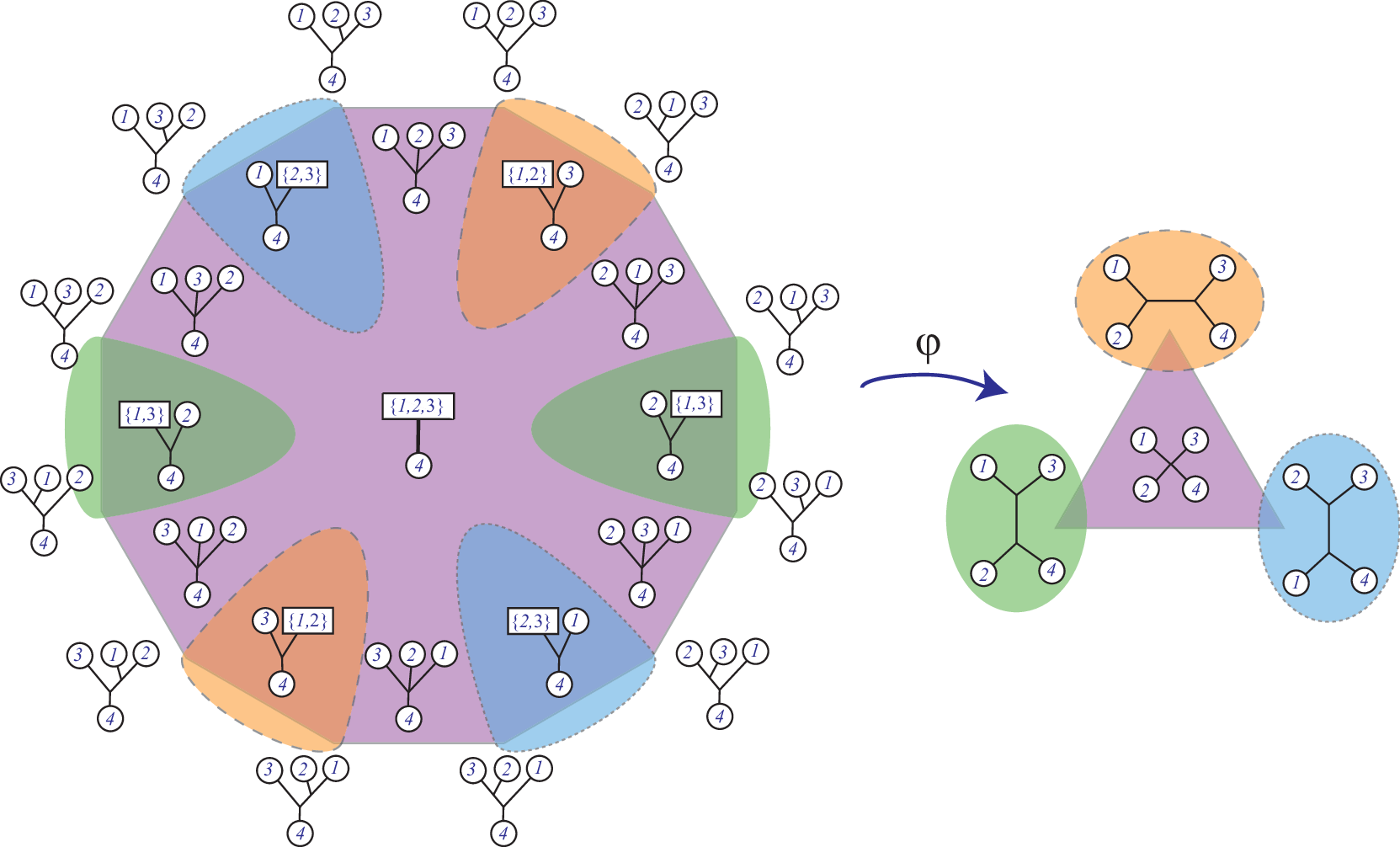}
\caption{Action of the map $\varphi$: the shaded faces all map to
the shaded vertices.}\label{f:2phi}
\end{figure}

Since our map preserves the face order, it takes vertices to
vertices. It is a set projection on vertices, and the number of
elements in a preimage has a nice formula:
\begin{prop} Let $T$ be a binary phylogenetic tree with $n$ leaves.
The number of ordered plane rooted binary trees $t$ such that
$\varphi(t) = T$ is $2^{n-2}.$
\end{prop}
\begin{proof} We note that the map is a surjection from vertices to vertices, since any leaf of a binary phylogenetic tree may be chosen as the root. By symmetry of the labeling of leaves, the size of
each preimage must be the same. Here $n$ is the total number of
leaves, so in the permutoassociahedron vertices we are considering
plane binary trees with $n-1$ leaves and a root. We divide the total
numbers of vertices of the two polytopes:
$$\frac{C_{n-2}(n-1)!}{(2n-5)!!} = 2^{n-2}.$$
Here we have used the formula for Catalan numbers: $C_{n-2} =
\frac{1}{n-1}{2(n-2) \choose n-2}.$ We also use the formula
$(2n-5)!! = \frac{(2n-4)!}{2^{n-2}(n-2)!}.$ \qed \end{proof}

Now we show how the targets of the map $\varphi$ are actually faces
of the BME polytope. Note that the image of the (improper) face
which is the entire permutoassociahedron (as well as any of its
corrolla facets) is the phylogenetic tree which is a corolla, or
star: it has only one node with degree $\ge 3$. This corolla
corresponds to the (improper) face which is the entire BME polytope.
In what follows we will assume we are speaking of proper faces.

\begin{theorem}\label{t:faces}
For each non-binary phylogenetic tree $t$ with $n$ leaves there is a
corresponding face $F(t)$ of the BME polytope $\p_n$. The vertices
of $F(t)$ are the binary phylogenetic trees which are refinements of
$t.$
\end{theorem}

\begin{proof}
We show that for each non-binary $t$ there is a distance vector
${\mathbf d}(t)$ for which the product ${\mathbf d}(t)\cdot {\mathbf
x}(t')$ is minimized simultaneously by precisely the set of binary
phylogenetic trees $t'$ which refine $t.$

The distance vector ${\mathbf d}(t)$ is defined as follows: the
component $d_{ij}(t)$ is the number of edges in the path between
leaf $i$ and leaf $j.$ Next we show that, for any tree $t'$, we have
the inequality:
$$\sum_{i<j}d_{ij}(t)x_{ij}(t') \ge 2^{(n-2)}|E(t)|$$
where $E(t)$ is the set of edges of $t.$ Moreover, we will show that
the inequality is precisely an equality if and only if the tree $t'$
is a refinement of $t.$

Our vector ${\mathbf d}(t)$ is constructed to be a vector of
distances (of paths between leaves) for any binary tree that refines
$t.$ This is seen by assigning a length of 1 to each edge of the
tree $t$, and calculating the distances between leaves by adding the
edge lengths on the path between them for any two leaves. A binary
tree $t'$ that refines $t$ is similarly given lengths of 1 for its
edges, except for those edges whose collapse would return $t'$ to
the tree $t.$ These latter edges are assigned a length of zero.

Now our result follows: given a distance vector whose components are
the distances between leaves on a binary tree, the dot product of
this vector with vertices of the BME polytope is minimized at the
vertex corresponding to that tree. In our case all the binary trees
$t'$ which refine $t$, with their assigned edge lengths, share the
distance vector ${\mathbf d}(t)$. Thus they are simultaneously the
minimizers of our product, and the value of that product is
$2^{(n-2)}$ times their common tree length. \qed \end{proof}

\begin{defi} For a non-binary phylogenetic tree $t$ we call the corresponding face of the BME
polytope the \emph{tree-face} $F(t).$
\end{defi}

An example of a tree-face, its vertices, and its inequality as given
in the proof of Theorem~\ref{t:faces}, are shown in
Fig.~\ref{f:examp}.

\begin{figure}[b!]\centering
                  \includegraphics[width=5in]{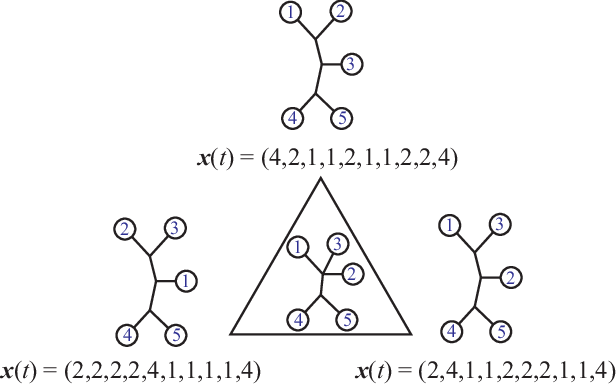}
\caption{The three binary trees shown are the vertices of the
tree-face corresponding to the tree in the center. The inequality
which defines this face is:
$2x_{12}+2x_{13}+3x_{14}+3x_{15}+2x_{23}+3x_{24}+3x_{25}+3x_{34}+3x_{35}+2x_{45}
\ge 48$ }\label{f:examp}
\end{figure}

Some special cases of tree faces are important. First we mention the
case in which the tree $t$ has only one non-binary node, that is,
exactly one node with degree larger than 3. Thus $t$ can be seen as
a collection of clades (and some single leaves) all attached to the
non-binary node.

\begin{prop}\label{t:cladeface}
For $t$ an $n$-leaved phylogenetic tree with exactly one node $\nu$ of degree
$m>3$, the tree-face $F(t)$ is precisely the clade-face $F_{C_1,\dots,C_p},$
defined in \cite{Rudy}, corresponding to the collection of clades
$C_1,\dots,C_p$ which result from deletion of $\nu.$ Thus $F(t)$ is
combinatorially equivalent to the smaller dimensional BME polytope $\p_m.$
\end{prop}
\begin{proof}
Any tree $t'$ which is a binary refinement of $t$ can be constructed
by attaching the clades $C_1,\dots,C_p$ to $p$ of the leaves of a
binary tree $\hat{t}$. Note that since we don't consider single
leaves to be clades, we need to say that $\hat{t}$ has $m$ leaves
where $m-p$ is the number of single leaves attached to $\nu.$

Recall from \cite{Rudy} that the face $F_{C_1,\dots,C_p}$ is the
image of an affine transformation of the BME polytope $\p_{m}.$ As
stated by those authors, this combinatorial equivalence follows
since every tree in $F_{C_1,\dots,C_p}$ can be constructed by
starting with a binary tree on $m$ leaves and attaching the clades
$F_{C_1,\dots,C_p}$ to $p$ of the $m$ leaves. \qed \end{proof} See
Fig.~\ref{f:examp} for an example of a clade-face, in fact a
\emph{cherry clade-face}, where the single clade in question is the
cherry $\{4,5\}$.

In \cite{Rudy} it is pointed out that the clade-faces form a
sub-lattice of the lattice of faces of $\p_n.$ Containment in that
sublattice is simply refinement, where a sub-clade-face of a clade-
face $F(t)$ can be found by refining the tree $t$, as long as the
result still has only a single non-binary node.

Now it is straightforward to see that refinement of trees in general
gives a partial ordering of tree-faces, and indeed another
sub-lattice of faces of the BME polytope which contains the
clade-faces as a sub-lattice. We note that the map $\varphi$ from
the permutoassociahedron is a lattice map.
\begin{prop}\label{t:lattice}
If $x \le y$  (thus containment as faces in the face lattice of
$\KP_n,)$ then $\varphi(x) \le \varphi(y)$ (so containment as faces
in the face lattice of $\p_n,$ the BME polytope).
 \end{prop}
 \begin{proof}
 The refinement of a labeled plane rooted tree $t$, or the refinement of the
 ordered partition labeling the leaves, both correspond to the
 refinement of $\varphi(t)$. The former is direct, the latter is seen
 via the replacement of parts in the partition by the corresponding
 corollas, before and after subdivision.
 \qed \end{proof}

Next we look at what are perhaps the most important tree-faces:
those which correspond to facets of the BME polytope. It turns out
that these facets correspond to trees $t$ which have exactly two
adjacent nodes with degree larger than 3.

\begin{theorem}\label{t:splitfacet}
Let  $t$ be a phylogenetic tree with $n>5$ leaves which has exactly
one interior edge $\{\nu, \mu\}$, with $\nu$ and $\mu$ each having
degree larger than 3. Then the trees which refine $t$ are the
vertices of a facet of the BME polytope $\p_n.$
\end{theorem}

The proof is in Section 6. Note that this implies that there are
clade-faces which are not contained in any tree-face facet, as seen
in Fig.~\ref{lattice}.

\begin{figure}[b!]\centering
                  \includegraphics[width=5in]{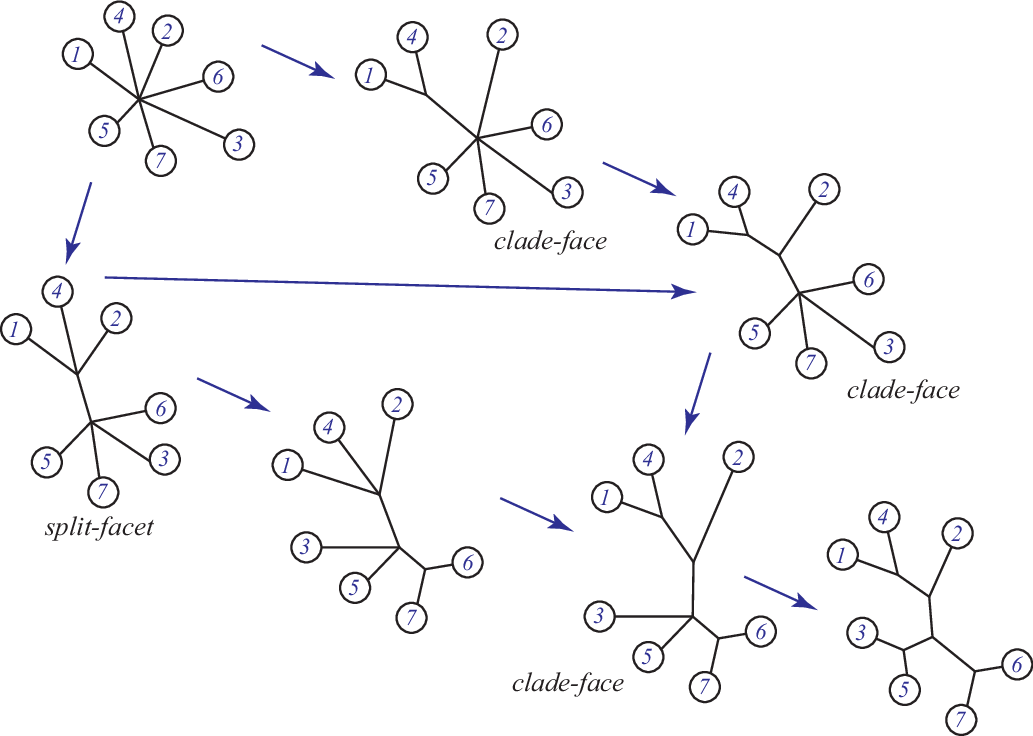}
\caption{Examples of chains in the lattice of tree-faces of the BME
polytope $\p_7.$ }\label{lattice}
\end{figure}

It is clarifying to refer to the new family of facets in
Theorem~\ref{t:splitfacet} as \emph{split-facets}. The binary
phylogenetic trees which display a given split correspond precisely
to the trees which refine a tree as described in that theorem.

In fact we can see all the tree-faces in terms of displayed splits,
since a split always corresponds to an internal edge. Thus  we have
that requiring 2 or more splits which the binary trees must all
display simultaneously corresponds to specifying a tree-face, all of
which are subfaces of split-facets.

\section{Enumeration}

\subsection{Number of split-facets.}
For $n=6$ there are 31 splits in all, but only $10$ splits which
obey the requirement that there are at least three leaves in each
part. For $n$ leaves the number of splits is $2^{(n-1)} -1.$ (This
is half the number of nontrivial, proper subsets.) Discarding the
splits with only one leaf and discarding the cherry clade-faces, we
are left with:
$$2^{(n-1)}-{n \choose 2}-n-1$$ split facets.

\subsection{Number of vertices in a split-facet.}
For $n=6$ each facet of this type has 9 vertices since there are
three choices of binary structure on each side of the split. Thus
the facet itself must be an $8$-dimensional simplex.

 We also found a formula for
the number of vertices in a split-face with parts of the split being
$S_1$ of size $k$ and $S_2$ of size $m= n-k$. The number of vertices
is:
$$(2m-3)!!(2k-3)!!~.$$ This formula is found via the multiplication
principle, in which all possible clades are counted for each part of
the split.

\subsection{Number of facets that a given tree belongs to, in the
Splitohedron.}

 The split-faces, intersecting-cherry facets,
and caterpillar facets together outline a relaxation of the BME
polytope. We define a new polytope:
\begin{defi}
The \emph{splitohedron} $Sp(n)$ is defined as the intersection of
the half-spaces of $\R^{n \choose 2}$ given by the following
inequalities listed by name: the intersecting-cherry facets, the
split-facets, the caterpillar facets and the cherry clade-faces--
and also obeying the $n$ Kraft equalities.
\end{defi}

The splitohedron is a bounded polytope because the cherry
clade-faces, where the inequality is $x_{ij} \le 2^{n-3}$, and the
caterpillar facets, where the inequality is $x_{ij} \ge 1$, show
that it lies inside the hypercube $[1,2^{n-3}]^{n \choose 2}$. It
has the same dimension as the BME polytope, and often has many of
the same vertices.

\begin{theorem} For an $n$-leaved binary phylogenetic tree, if the number of
cherries is at least $n/4$ then the tree represents a vertex in the
BME polytope that is also a vertex of the splitohedron. For $n\le
11$ the tree represents a vertex regardless of the number of
cherries.
\end{theorem}
\begin{proof}
For a given binary phylogenetic tree $t$ it is straightforward to
count how many distinct facets of the splitohedron it belongs to. If
that number is as large as the dimension, we know that the tree lies
at a vertex of the polytope $Sp(n)$. First we note that an
inequality which defines a facet of the BME polytope and which is
also obeyed by the splitohedron therefore defines a facet of the
splitohedron as well, by the nature of relaxation.

For each cherry $\{a,b\}$ of $t$ we have that $t$ lies within
$2(n-2)$ facets, an intersecting-cherry facet for each choice of
either $a$ or $b$ and a third leaf that is neither. For each
interior edge that does not determine a cherry clade, we have that
$t$ lies within a split-facet. There are $n-3-p$ such interior
edges, where $p$ is the number of cherries. Finally, if $t$ is a
caterpillar then it lies within 4 caterpillar facets, determined by
a choice of one leaf from each cherry to fix.

All together $t$ lies within $p(2n-4) + n - 3 - p =(2n-5)p + n -3$
facets of the splitohedron, if it is not a caterpillar. For any $n$
this number increases with $p$, as $p$ ranges from 2 to $\left
\lfloor{\frac{n}{2}}\right \rfloor.$ The dimension of the polytope
is ${n \choose 2} -n = \frac{1}{2}(n^2-3n).$  Comparing the two
expressions shows that the tree $t$ will represent a vertex of
$Sp(n)$ as long as $p \ge \frac{n^2-5n+6}{4n-10}. $ This is true,
for instance, when $p \ge n/4.$

In the worst case scenario for non-caterpillar trees, we have $p=3$
and $t$ is a vertex when $n^2 \le 17n -36,$ or for $n \le 14.$ For
caterpillar trees, where $p=2$, we have the extra four facets so $t$
is a vertex when $n^2 \le 13n -22,$ or for $n \le 11.$

Thus for $n \le 11$ we have all the binary phylogenetic trees
represented as vertices of the splitohedron. \qed \end{proof}

\section{Proof of Theorem~\ref{t:splitfacet}}

First we prove that the split-facet is always a face of the BME
polytope. This is implied by Theorem~\ref{t:faces}. However it is
more useful to prove the following simpler linear inequality.
\begin{lemma}Consider the split $\pi = \{S_1, S_2\}$ of the set of leaves.
Let $|S_1| = k \ge 3$ and $|S_2| = m \ge 3.$ Then the following
inequality becomes an equality precisely for the trees which display
the split, and a strict inequality for all others.
$$\sum_{\text {$i<j$,
 leaves $i,j\in S_1$}} x_{ij} \le (k-1)2^{n-3}.$$
\end{lemma}
\begin{proof}~(of the face inequality.) It follows directly from the fact that the sum of all coordinates for
 any tree with $n$ leaves is $n2^{n-3}.$ Thus, if we double-sum over the leaves,
  we have $\displaystyle{\sum_{i}(\sum_{j}x_{ij}) = n2^{n-2}}$; twice the total since
   we add each coordinate twice. Now consider a tree with $k+1$ leaves (anticipating a clade with $k$ leaves)
    and the double sum is $(k+1)2^{k-1}.$
 If we only sum over the first $k$ leaves, thereby ignoring all the coordinates involving
  the $k+1$st leaf, the smaller double sum totals to  $(k-1)2^{k-1}$. (Note that the additional
  internal node connecting to the $k+1$st leaf is causing the perceived difference in
  results for our clade of $k$ leaves from an entire tree of $n$ leaves. )
  Next consider the actual situation of interest, where
   there is a clade of size $k$ whose coordinates we double-sum over,
     but we have replaced the extra leaf with another clade of size $n-k$ . Here each
     coordinate in the double sum is multiplied by the power of 2 achieved by
      adding $n-k-1$ leaves, so our total becomes $$2^{n-k-1}(k-1)2^{k-1} = (k-1)2^{n-2}.$$

 Recall that we have been double counting, so our result is 2 times too much:
  the actual sum of the coordinates in any clade with leaves from $S_1$ is $$\sum_{\text {$i<j$,
 leaves $i,j\in S_1$}} x_{ij} =
(k-1)2^{n-3}.$$

  It is clear that for any tree which does not contain a clade consisting
   only of the leaves in $S_1,$ it instead must contain a collection of
    clades whose leaves together make up the set $S_1$ (some of which may be singletons.)
     Since some of these must be further apart (separated by more internal nodes from each other) than if they formed a single clade, then
     summing all the coordinates using indices only from $S_1$  will give a total strictly smaller
      than in the case where $S_1$ makes up the leaves of a single clade.  \qed \end{proof}

\begin{figure}[b!]\centering
                  \includegraphics[width=\textwidth]{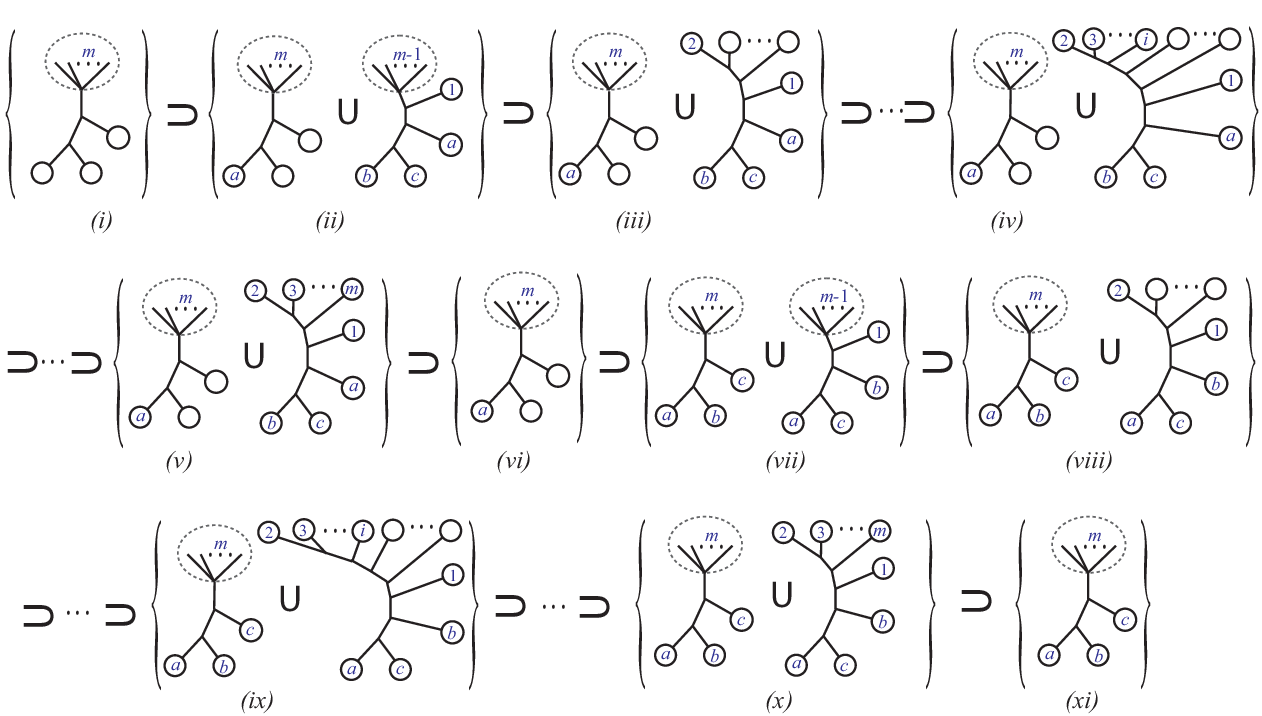}
\caption{Flag for the base case in proof of
Theorem~\ref{t:splitfacet}. The sets include all the trees that can
be formed by completing the pictures with additional leaf labels.
Dashed-circled corollas denote all possible binary structures on the
leaves (which are not always shown). Dots between labeled leaves
denote an ordered caterpillar structure, while dots between
unlabeled leaves denote an unordered caterpillar. }\label{flag_base}
\end{figure}

\begin{figure}[b!]\centering
                  \includegraphics[width=\textwidth]{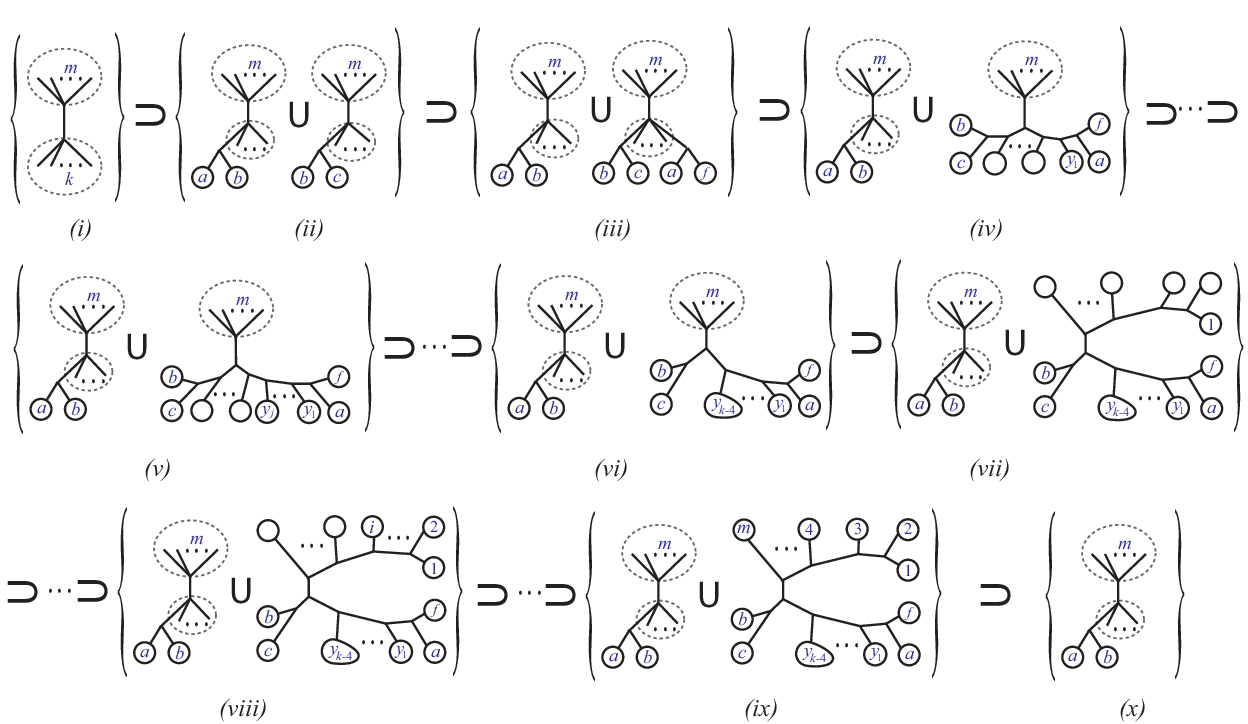}
\caption{Flag for the inductive step in proof of
Theorem~\ref{t:splitfacet}. Picture notation is as above.
}\label{flag_induct}
\end{figure}

Notice that using the second part of the split, $S_2,$ as the basis
for the sum works just as well. In practice the smaller part of the
split is chosen in order to provide a shorter inequality. Now we
prove the dimension of these faces.
\begin{proof}~of Theorem~\ref{t:splitfacet}:
\emph{Base case.} The proof is inductive. We start by proving the base case in
which one of the parts of the split has exactly $k=3$ leaves, and the other has
size $m\ge 3$.

 To do this, we
fill in the flag which goes from this facet down to the clade face
for a fixed combination of the $3$-leaved section of the split.

The first inequality is that of the facet itself, where we simply
have a split. If we label the leaves in our $k=3$-leaf section
$a$,$b$,$c$; then our simplified inequality from above is
$x_{a,b}+x_{a,c}+x_{b,c} \leq 2^{n-2}$. Let the leaves in the
$m$-leaf section be labeled as $1,2,...,m$. We now rely on the fact
that to show a chain of subfaces, our subsequent face inequalities
only need to be strict on trees  which obey the previous face
inequality exactly, as an equality. This raises a caveat: the
inequalities used for subfaces of the flag in our proof may not be
actual face inequalities of the entire polytope.

Our next inequality is: $$3x_{a,1}-x_{b,1}-x_{c,1}+2x_{a,b}+2x_{a,c}
\leq 3\cdot 2^{n-3}.$$ This is intended to include all trees with
$a$ in a cherry, and to require the leaf $1$ to be near the leaf $a$
when $a$ is not in the cherry. See the set pictured in part $(ii)$
of Fig.~\ref{flag_base}.

In the case when $a$ is in the cherry,  $x_{b,1}$ or $ x_{c,1}$ will
be the size of $x_{a,1}$ and the other will be twice its size. So
the sum $3x_{a,1}-x_{b,1}-x_{c,1}$ will be 0. Then, $x_{a,b}$ or
$x_{a,c}$ must be $2^{n-3}$ and the other $2^{n-4}$. These add to
$3\cdot 2^{n-4}$. So $2x_{a,b}+2x_{a,c} = 3\cdot 2^{n-3}$.

When $a$ is not in the cherry, for out inequality to be maximal we
must have $x_{a,1}=2^{n-4}$ and hence $x_{b,1}$ and $x_{c,1}$ as
$2^{n-5}$. So $3x_{a,1}-x_{b,1}-x_{c,1}=3\cdot 2^{n-4}-2\cdot
2^{n-5}=2^{n-3}$. Then, since $a$ is near $b$ and $c$ but not in the
cherry, we have $2x_{a,b}+2x_{a,c}=2\cdot 2^{n-3}$. So, the left
hand side of our equation is $3\cdot 2^{n-3}$ when $1$ is close to
$a$, as wanted. If $1$ were to be further, it is easy to see the
expression would be smaller.

Our next set of steps is dependent upon the size of $m$. The intent
here is to build off of previous steps by forcing specific leaves to
be far from the $k$-leaf cluster in each step. See the sets pictured
in parts $(iii)-(v)$ of Fig.~\ref{flag_base}. Our inequalities will
be:
$$3x_{a,i}-x_{b,i}-x_{c,i}+\dfrac{2^{i-1}}{2^{n-4}}(x_{a,b}+x_{a,c}) \geq
3\cdot 2^{i-1}$$ when $i\geq3.$ When $i=2$, we use the inequality:
$$3x_{a,2}-x_{b,2}-x_{c,2}+\dfrac{2^{3-1}}{2^{n-4}}(x_{a,b}+x_{a,c}) \geq
3\cdot 2^{3-1}.$$ This is because $2$ is in a cherry with $3$ so
they must satisfy the same inequality, albeit with different
coordinates.

This works since $3x_{a,i}-x_{b,i}-x_{c,i}$ is $0$ when $a$ is in
the cherry, and it is half the size of
$\dfrac{2^{i-1}}{2^{n-4}}(x_{a,b}+x_{a,c})$ when it is not in the
cherry. Also, $\dfrac{2^{i-1}}{2^{n-4}}(x_{a,b}+x_{a,c})$ is
$\frac{3}{2}$ the size when $a$ is in the cherry as when $a$ is not
in the cherry. So in both cases, when we have what we want, we have
equality. If the leaf $i$ moves at all when $a$ is in the cherry, we
still have equality. If it moves when $a$ is not in the cherry,
$3x_{a,i}-x_{b,i}-x_{c,i}$ will become larger.

After this chain, we have a simple inequality which forces $a$ to be
in the cherry, as in the set pictured in part $(vi)$ of
Fig.~\ref{flag_base}. It looks like: $$2x_{a,b}+2x_{a,c}\leq 3\cdot
2^{n-3}.$$

Next, for the set pictured in part $(vii)$ of Fig.~\ref{flag_base},
we have $3x_{b,1}-x_{a,1}-x_{c,1}+2x_{a,b}+2x_{b,c} \leq 3\cdot
2^{n-3}$. This works like the inequality for the face below the
facet. This meaning that, it forces $1$ to be close to $b$ when $b$
is not in the cherry, and has no effect on the tree when $b$ is in
the cherry. We then have the same $i$-indexed chain after it with
the roles of $a$ and $b$ reversed, since we are trying to achieve
the same result as with $a$ but with $b$. See the sets pictured in
part $(viii) - (x)$ of Fig.~\ref{flag_base}. So, the inequalities
are:
$$3x_{b,i}-x_{a,i}-x_{c,i}+\dfrac{2^{i-1}}{2^{n-4}}(x_{b,a}+x_{b,c}) \geq
3\cdot 2^{i-1}$$ when $i\geq3$ and when $i=2$,
$$3x_{b,2}-x_{a,2}-x_{c,2}+\dfrac{2^{3-1}}{2^{n-4}}(x_{b,a}+x_{b,c}) \geq
3\cdot 2^{3-1}.$$

To finish, we  use the fixed clade face of dimension
$\binom{m+1}{2}-(m+1)$ as described in \cite{Rudy} where $c$ is not
in the cherry. See the set pictured in part $(xi)$ of
Fig.~\ref{flag_base}. The total length of our chain is
$\binom{n}{2}-n-1$, proving that the (m,3)-split face is a facet.

\emph{Inductive step.} Next we assume the theorem for splits of respective
sizes $k-1$ and $m,$ both larger than 3, and inductively prove it for all
$k,m.$ We consider all the trees which display a given split $\pi$ into leaves
$S_1=\{1,\dots,m\}$ and leaves $S_2=\{a,b,c,f\}\cup\{y_1,\dots,y_{k-4}\}.$

The inductive assumption allows us to use Theorem~\ref{t:cladeface}
in our proof. We can calculate the dimension of a face which has as
its vertices all the binary phylogenetic trees that both display the
split $\pi$ and also have a cherry $\{a,b\}.$ These trees are a
subset of the set of all the trees with the cherry $\{a,b\}$, which
describes a clade-face of the BME polytope. That clade-face is
equivalent to $\p_{n-1},$ and using the argument of the proof of
Theorem~\ref{t:cladeface} as found in \cite{Rudy}, the cherry can be
considered as a leaf of the trees of $\p_{n-1}.$ Thus the trees that
both display our split and also have a cherry $\{a,b\}$ display a
split $\pi'$ into $m$ and $k-1$ ``leaves'' which gives, by
induction, a facet of $\p_{n-1}.$ The dimension of this face is thus
${n-1 \choose 2} -(n-1) -1,$ and so it is the top-dimensional face
in a flag of length ${n-1 \choose 2} -(n-1).$

Next we show the existence of a chain of faces of length $n-2,$
beginning with the face of all trees that display our split $\pi$
and ending with the face that has all trees displaying $\pi$ and
possessing the cherry $\{a,b\}.$ Concatenating this chain to the
flag shown by induction gives a flag of length ${n \choose 2} -n$,
which implies that our split-face is indeed a facet.

After the split face, the second face in our flag is described by
all the trees that both display the split $\pi$ and possess either
cherry $\{a,b\}$ or cherry $\{b,c\}.$ These trees, as a sub-face of
the split face, have the face inequality: $$x_{ab}+x_{bc}-x_{ac} \le
2^{(n-3)}.$$ Note that this is a face by virtue of being the
intersection of the split-facet and the intersecting-cherry facet:
in fact the proof from here is inspired by the proof of Theorem 4
(the intersecting-cherry facet) in \cite{forcey2015facets}. Indeed
the next face in our flag is described by containing the trees which
both display the split $\pi$ and possess either cherry $\{a,b\}$ or
the two cherries $\{b,c\}$ and $\{a,f\}.$ Again this is an
intersection of faces: the split-face and the second face of the
flag shown in the proof of Theorem 4 in \cite{forcey2015facets}. For
completeness, the inequality obeyed by this third face is:
$$x_{bc}+x_{bf}-x_{ac} -x_{af} \ge 0.$$

Next we have a chain of $k-4$ faces which correspond to ordering the
remaining $k-4$ leaves of $S_2.$ For $j\in 1\dots k-4$ we take the
set of trees that have the split $\pi$ and the cherry $\{a,b\}$, or
which have the cherries $\{b,c\}$ and $\{a,f\}$ as the two cherries
of a caterpillar clade made from $S_2,$ and for which the leaves
$y_1 \dots y_j$ are attached in that order starting as close as
possible to the cherry $\{a,f\}.$ See the pictures of sets $(iv)$ -
$(vi)$ in Fig.~\ref{flag_induct}, noting how the caterpillar clade
is attached to $S_1$ at any point among its unordered nodes. The
$j^{th}$ term in this list of faces obeys the inequality:
$$(2^{n-3}-2^{m-1})(x_{ay_j}-x_{by_j})\le(2^{n-3}-x_{ab})(2^{n-3-j}-2^{m+j-1}).$$
To see that this is an equality for the sets of trees in question,
note first that when $\{a,b\}$ is a cherry then $x_{ab} = 2^{n-3}$
and $x_{ay_j}=x_{by_j}.$ Also, when $S_2$ is fixed as a caterpillar
clade, then $x_{ab} = 2^{m-1} $ and $x_{ay_j}-x_{by_j} =
2^{n-3-j}-2^{m+j-1}.$ Finally, when $y_j$ is found in a location on
the caterpillar clade closer to leaf $b,$ (which is the only way to
be in the previous face while avoiding being in the current face),
then $x_{ay_j}-x_{by_j}$ is forced to be a lesser value.

After the chain of caterpillar clades using $S_2$, we add a chain
using caterpillar clades on $S_1.$ This chain begins with the set
pictured in part $(viii)$ of Fig.~\ref{flag_induct}, where the leaf
1 is in the cherry at the far end of the flag.  This face obeys the
equality:
$$(2^{n-3}-2^{m-1})(x_{a1}-x_{b1})\le(2^{n-3}-x_{ab})(2-2^{k-3}).$$
The comments just made about the previous faces also apply here, to
show that the equality holds on the face and that when $S_2$ is
fixed as a caterpillar clade, then $x_{ab} = 2^{m-1}$. Now though we
see that if the leaf 1 is any closer to $b,$ then both $x_{a1}$ and
$x_{b1}$ increase. However, since they are both powers of two then
increasing both by a factor of another power of two means their
difference will be even larger--and we are subtracting in the order
that ensures the inequality.

The remaining links in the chain are formed by fixing the leaves
$2,\dots,m$ in order along the caterpillar clade in $S_1,$ as in
pictured sets $(viii)$ and $(ix)$ of Fig.~\ref{flag_induct}. When
the leaf $i$ is fixed, the face obeys the inequality:
$$(2^{n-3}-2^{m-1})(x_{ai}-x_{bi})\le(2^{n-3}-x_{ab})(2^{i-1}-2^{i+k-5}).$$
This inequality is an equality on the face and strict on the trees
of the previous face excluded from the current face, by the same
arguments as above.

Finally we exclude all the trees displaying the split except for
those with the cherry $\{a,b\},$ as shown by the pictured set $(x)$
in Fig.~\ref{flag_induct}. This completes the proof by induction, as
explained above. \qed \end{proof}

\section{Future work}
  We have shown that (for $n\le 11$) the
splitohedron contains among its vertices all the possible
phylogenetic trees. Therefore if the BME linear program is optimized
in the splitohedron at a valid tree vertex for $n\le 11$, it is also
optimized in the BME polytope.

More importantly, however, the binary phylogenetic trees for any $n$
all lie on the boundary of several facets of the splitohedron which
are also facets of the BME polytope. Our continuing research program
involves writing code that uses various linear programming methods
in sequence, with a branch-and-bound scheme, to find the BME tree.

 Then by finding further facets we
will improve this theorem, hopefully to a version that holds for all
$n>11$.



\section{Acknowledgements}
We thank the editors and both referees for helpful comments. The
first author would like to thank the organizers and participants in
the Working group for geometric approaches to phylogenetic tree
reconstructions, at the NSF/CBMS Conference on Mathematical
Phylogeny held at Winthrop University in June-July 2014. Especially
helpful were conversations with Ruriko Yoshida, Terrell Hodge and
Matt Macauley. The first author would also like to thank the
American Mathematical Society and the Mathematical Sciences Program
of the National Security Agency for supporting this research through
grant H98230-14-0121.\footnote{This manuscript is submitted for
publication with the understanding that the United States Government
is authorized to reproduce and distribute reprints.} The first
author's specific position on the NSA is published in
\cite{freedom}. Suffice it to say here that he appreciates NSA
funding for open research and education, but encourages reformers of
the NSA who are working to ensure that protections of civil
liberties keep pace with intelligence capabilities.

\bibliography{phylogenetics}{}
\bibliographystyle{plain}

\end{document}